\documentclass{lmcs} 
\usepackage[utf8]{inputenc}
\pdfoutput=1

\usepackage{lastpage}
\lmcsdoi{22}{4}{19}
\lmcsheading{}{\pageref{LastPage}}{}{}%
{Jan.~12,~2026}{Nov.~09,~2026}{}


\usepackage{amsmath,amssymb,amsthm}

\usepackage{tikz}
\usetikzlibrary{positioning}
\usepackage{graphicx}
\usepackage{caption}
\usepackage[T1]{fontenc}

\usepackage{xcolor} 
\usepackage{url}
\usepackage[hidelinks]{hyperref}
\usepackage[capitalize]{cleveref}

\newcommand{\N}{\mathbb N}

\newcommand{\Q}{\mathbb Q}

\newcommand{\Bai}{\N^{\N}}
\newcommand{\Seq}{\N^{<\omega}}

\newcommand{\Can}{2^{\N}}

\newcommand{\bsigma}{\boldsymbol\Sigma}
\newcommand{\bpi}{\boldsymbol\Pi}
\newcommand{\bdelta}{\boldsymbol\Delta}

\newcommand{\lnh}{{\rm{length}}}
\newcommand{\card}{{\rm{card}}}

\newcommand{\quot}[2]{{\raisebox{.2em}{$#1\!$}\left/\raisebox{-.2em}{$#2$}\right.}}

\newcommand{\Cl}[1]{\mathsf{Cl}(#1)}
\newcommand{\Int}[1]{\mathsf{Int}(#1)}
\newcommand{\Fr}[1]{\mathsf{Fr}(#1)}

\usepackage{hyperref}

\def\eg{{\em e.g.}}
\def\ie{{\em i.e.}}

\begin{document}

\title{The complexity of being monitorable} 

\author[R.~Camerlo]{Riccardo Camerlo}[a]
\author[F.~Dagnino]{Francesco Dagnino\lmcsorcid{0000-0003-3599-3535}}[b] 

\address{DIMA - Universit\`a di Genova}	
\email{riccardo.camerlo@unige.it}  

\address{DIBRIS - Universit\`a di Genova}	
\email{francesco.dagnino@unige.it}  




\begin{abstract}
\noindent We study monitorable sets from a topological standpoint.
In particular, we use descriptive set theory to describe the complexity of the family of monitorable sets in a countable space $X$.
When $X$ is second countable, we observe that the family of monitorable sets is $\bpi^0_3$ and determine the exact complexities it can have.
In contrast, we show that if $X$ is not second countable then the family of monitorable sets can be much more complex, giving an example where it is $ \bpi^1_1$-complete.
\end{abstract}

\maketitle


\section{Introduction}

Runtime Verification (RV)~\cite{LeuckerS09,BartocciFFR18} 
is a formal verification technique that, differently from others such as Model Checking,
focuses on the analysis of actual system executions, rather than on a model (that is, an abstraction) of it.
More precisely, 
given a formal specification or property of the expected system behaviours, 
one synthesises a \emph{monitor}, that is, a piece of software that observes system executions and tries to establish whether they satisfy or violate the specification. 
This makes RV a lightweight and scalable approach to formal verification at the price of losing exhaustiveness: 
the verdict a monitor can reach is valid only for the executions it has observed, not for all possible ones. 

Since in many cases monitors have to run together with the system they observe, 
one of the key challenges in RV is to ensure that the monitor runs only if it has the possibility of eventually concluding a verdict, thus avoiding a waste of resources. 
Unfortunately, one soon realises that there are properties whose associated monitor could end up in situations where, no matter what it observes, it would never be able to determine the satisfaction or the violation of the property. 
This has led to the notion of \emph{monitorability}~\cite{PnueliZ06}.
Roughly, a property is said to be \emph{monitorable} 
if we can synthesise a monitor which has always the possibility of eventually verifying the satisfaction or the violation of the property, that is, 
the above situation cannot happen. (Actually, there are various flavours of monitorability, depending on the operational guarantees one wants to obtain on monitors; we refer to \cite{AcetoAFIL21ssm} for a discussion of the whole spectrum.) 

Monitorability is typically studied at the syntactic level separately for each specification language, see \eg, \cite{BauerLS11,FrancalanzaAI17,AmaraBFF25}. 
This has led to a scattered family of results with several concrete applications that share many similarities but lack a unifying picture.
In order to achieve a general and syntax independent treatment of monitorability,
Diekert and Leucker~\cite{dieker2014} propose to take a more abstract semantic perspective, studying monitorability with topological tools. 
The key insight is that possible system behaviours can be viewed as points in a topological space where open subsets abstractly model the \emph{observations} a monitor can perform. 
Within this framework, a property is simply a subset of the space—\ie, a set of valid behaviours.
Therefore, given a topological space $X$, 
monitorable properties determine a subset $\mathcal{M}(X) \subseteq \mathcal{P}(X)$, where $\mathcal{P}(X)$ denotes the powerset of $X$, \ie, the collection of all properties on $X$. 

The purpose of this paper is to determine how difficult it is establishing whether a property is monitorable or not. 
To this end, we study the descriptive complexity of $\mathcal M (X)$ as a subset of $\mathcal{P}(X)$ when $X$ is a countable space. 
Indeed, in this case the powerset $\mathcal{P}(X)$ can be endowed with a topology homeomorphic to the Cantor space, so 
one can use the tools and techniques of descriptive set theory \cite{Kechris1995} to determine the topological complexity of $\mathcal{M}(X)$. 
We do this in the same vein as \cite{Todorc2014}, where Todor\v cevi\'c and Uzc\'ategui studied the complexity of topologies as subsets of $ \mathcal P (X)$.

\subsubsection*{Outline} 
Section~\ref{sec:basics-mon} introduces monitorable subsets of a topological space together with some of their basic properties. 

The core of the paper is section \ref{sec:cntbl}, which studies monitorable sets in countable topological spaces.
Section \ref{sec:basics-top} recalls the notions in descriptive set theory that are used in this section. 
In Section \ref{sec:sccnt} we determine  the possible topological complexities, measured using Wadge reducibility, for the family of monitorable subsets in a countable, second countable space. 
The result is represented in figure \ref{fig:res}.
Thus, the family of monitorable subsets of a countable, second countable space has always moderate complexity, being $ \bpi^0_3$: the results of section \ref{sec:sccnt} give in particular a criterion to recognise whether such a family is of the highest possible complexity, namely $ \bpi^0_3$-complete.
We apply this characterisation in section \ref{sec:trn} to topologies induced by a transition relation of an automaton: we show that most (in the sense of Baire category) transition relations on a countable set of states give rise to a topology whose family of monitorable sets is topologically simple, that is, $ \bsigma^0_2$.

The situation for non-second countable topologies is quite different: in fact, in section \ref{sec:nsc} we give an example of a countable, non-second countable space whose family of monitorable sets is complete coanalytic.

\section{Monitorable sets: definition and first properties}
\label{sec:basics-mon}

Diekert and Leucker~\cite{dieker2014} have introduced a topological formalisation of the notion of monitorability.
In this section we recall their definition and provide some basic properties of monitorable sets. 

The terminology and notation we employ are largely standard and self-explanatory.
We denote $A\triangle B$ the symmetric difference of the sets $A,B$.
For $X$ a topological space and $A\subseteq X$, we write 
$\Cl{A}$ for the closure of $A$, 
$\Int{A}$ for the interior of $A$, and 
$\Fr{A}$ for the frontier of $A$, that is, $\Cl{A}\setminus\Int{A}$.
A basis for the topology always consists of non-empty sets.
A point $x\in X$ is \emph{isolated} if $\{ x\} $ is open.
If $U$ is an open subset of $X$, we say that $A$ is \emph{dense} in $U$ if $U\subseteq \Cl{A} $; similarly, we say that $A$ is \emph{codense} in $U$ if $U\subseteq \Cl{X\setminus A} $.
We say that $A$ is \emph{nowhere dense} if $\Int{\Cl{A}} = \emptyset$, \ie, the only open subset in which $A$ is dense is the empty one.
An open set $U$ is \emph{regular} if $U= \Int{ \Cl{U} } $.

\begin{defiC}[\cite{dieker2014}]
Let $X$ be a topological space.
A subset $A\subseteq X$ is \emph{monitorable} if for every non-empty, open set $U\subseteq X$ either $A$ is not dense in $U$ or $A$ is not codense in $U$.
\end{defiC}

\noindent
In other words, a subset $A$ of a topological space $X$ is monitorable if and only if, for every non-empty open subset $U\subseteq X$, there is a non-empty open subset $V\subseteq X$ such that 
either $V\subseteq U\cap A$ or $V\subseteq U \setminus A$. 
Intuitively, this means that every observation $U$ can be refined to another observation $V\subseteq U$  
which is either a subset of $A$ or of its complement, \ie, $V$ determines either the satisfaction or the violation of $A$. 

We denote by $ \mathcal M (X)$ the family of monitorable subsets of $X$.
This is a subset of $ \mathcal P (X)$, the powerset of $X$, endowed with the product topology---after the identification of each subset of $X$ with its characteristic function.
If different topologies are considered on the same set $X$, the more precise notation $ \mathcal M ( \mathcal T )$ is used to denote the family of monitorable subsets of $X$ in the topology $ \mathcal T $.

The following proposition reports facts proved in \cite{dieker2014}.

\begin{prop} \label{prop:beg}
Let $X$ be a topological space.
Then:
\begin{enumerate}
\item\label{prop:beg:1} Open subsets of $X$ are monitorable
\item\label{prop:beg:2} If $A,B$ are monitorable subsets of $X$ then $A\cup B$ is monitorable
\item\label{prop:beg:3} If $A$ is a monitorable subset of $X$ then $X\setminus A$ is monitorable
\item\label{prop:beg:4} A subset $A$ of $X$ is monitorable if and only if $ \Int{ \Fr{A} }=\emptyset $
\end{enumerate}
\end{prop}

\begin{rem}
\leavevmode
\begin{enumerate}
\item From proposition \ref{prop:beg} it follows in particular that $ \mathcal M (X)$ is a Boolean algebra and that every nowhere dense set is monitorable.
\item In general, monitorable sets are not closed under countable union: in $ \Q $ every singleton is monitorable, but there are non-monitorable subsets (see for instance lemma \ref{lem:conditionpi}).
However, in \cite[theorem 1]{dieker2014} it is proved that if $X$ is completely metrisable and $A\subseteq X$ is both an at most countable union of monitorable sets and an at most countable intersection of monitorable sets, then $A$ is monitorable.
\end{enumerate}
\end{rem}

\begin{exa} \label{ex:finstr} 
Let $E$ be an at most countable set with at least two elements. 
Consider the set $E^{{<}\omega}$ of finite strings over $E$ endowed with the Alexandroff topology induced by the prefix ordering $\subseteq $, that is, $u\subseteq v$ iff there is $w\in E^{{<}\omega}$ such that $uw = v$. 
Hence, open subsets are unions of subsets of the form 
$u^{\uparrow }= \{ v \in E^{{<}\omega}\mid u \subseteq v\}$ 
for $u \in E^{{<}\omega}$.
For $A\subseteq E^{{<}\omega}$ we have that the closure of $A$ is given by 
$\Cl{A} = \{ v\in E^{{<}\omega} \mid v\subseteq u \mbox{ for some } u \in A \}$.
Fix two distinct elements $a,b\in E$ and consider the following subsets:
\begin{itemize}
\item $L_1 = \{ u \in E^{{<}\omega} \mid u = v_1 a v_2 \mbox{ for some } v_1,v_2\in E^{{<}\omega} \}$, that is, the set of all strings containing at least one occurrence of $a$
\item $L_2 = \{ u \in E^{{<}\omega} \mid u = b^n \mbox{ for some } n \in \N \}$, that is, the set of strings containing only occurrences of $b$ 
\item $L_3 = \{ u \in E^{{<}\omega} \mid u = a^n b^n \mbox{ for some } n \in \N \}$, that is,  the set of strings consisting of a string of $a$ followed by a string of $b$ of the same length
\item $L_4 = \{ u \in E^{{<}\omega} \mid u = vb \mbox{ for some } v \in E^{{<}\omega} \}$, that is, the set of strings ending with an occurrence of $b$
\end{itemize}
By definition, a subset $A\subseteq E^{{<}\omega}$ is monitorable if and only if 
for all $u \in E^{{<}\omega}$ there is $v\in E^{{<}\omega}$ such that 
either $(uv)^{\uparrow }\subseteq A$ or $(uv)^{\uparrow }\subseteq E^{{<}\omega}\setminus A$. 
Therefore, we have that $L_1$, $L_2$ and $L_3$ are monitorable, while $L_4$ is not. 
Indeed, $L_1$ is open, $L_2$ is closed, $L_3$ is nowhere dense and $L_4$ is both dense and codense in $E^{{<}\omega}$. 

This space abstractly models a typical setting for Runtime Verification: 
the set $E$  contains the events that a system can generate, 
executions of such a system are represented as sequences of events, i.e., elements of $E^{{<}\omega}$, and 
properties of the system are collections of valid executions, i.e., subsets of $E^{{<}\omega}$. 
The topology we have considered models the fact that 
monitors can only see finite prefixes of actual executions and, based on those, they should determine the satisfaction or violation of  the properties they are verifying. 
\end{exa}

\noindent
Denote $ \mathcal N (X)$ the ideal of nowhere dense subsets of $X$ and $ \mathcal R \mathcal O (X)$ the Boolean algebra of regular open subsets of $X$. 
We now show how the sets $\mathcal{M}(X)$, $\mathcal{N}(X)$ and $\mathcal{R} \mathcal O (X)$ are related. 
First, we have the following corollary to proposition \ref{prop:beg}. 

\begin{cor} \label{02072025}
The Boolean algebra $ \mathcal M (X)$ is the Boolean algebra generated by the open sets and the nowhere dense sets.
In fact, if $M\in \mathcal M (X)$, then there exist an open set $U$ and a set $N\in \mathcal N (X)$ such that $M=U\cup N$.
\end{cor}

\begin{proof}
Notice that $M= \Int{M} \cup ( \Fr{M} \cap M)$ and $ \Fr{M} \cap M$ is nowhere dense by proposition \ref{prop:beg}(4).
\end{proof}

\begin{rem}
\leavevmode
\begin{enumerate}
\item If $ \mathcal B $ is a basis for $X$, proposition \ref{prop:beg}(4) says that, given $A\subseteq X$,
\begin{equation} \label{eq:charact}
A\in \mathcal M (X)\Leftrightarrow\neg\exists U\in \mathcal B \ (U\subseteq \Cl{A} \land U\subseteq \Cl{X\setminus A} )
\end{equation}
\item Proposition \ref{prop:beg}(4) suggests that the regular open set $ \Int{ \Fr{A} }$ (or any notion of size on it) provides a way to assign a quantitative measurement of the non-monitorability of a set $A$.
Indeed, $ \Int{\Fr{A} } $ is the set of points at which $A$ is not monitorable according to the terminology of \cite[definition 2(1) and proposition 3]{dieker2014}.
\item Proposition \ref{prop:beg}(4) also shows that in most spaces there exist monitorable sets that are arbitrarily complex.
In fact, if $N$ is nowhere dense in $X$, then every subset of $N$ is in $ \mathcal M (X)$.
\end{enumerate}
\end{rem}

\noindent
Define an equivalence relation $E$ on $ \mathcal P (X)$ by letting
\[
AEB\Leftrightarrow A\triangle B\in \mathcal N (X)
\]
We denote $[A]_E$ the $E$-equivalence class of the set $A$.
Notice that $ \mathcal M (X)$ is $E$-invariant.

\begin{prop} \label{05072023}
Let $A\subseteq X$.
The following are equivalent:
\begin{enumerate}
\item $A$ is monitorable
\item $[A]_E$ contains some open set
\item $[A]_E$ contains exactly one regular open set
\end{enumerate}
The regular open set in (3) is $ \Int{ \Cl{A} } $ and is the biggest open set belonging to $[A]_E$.
\end{prop}

\begin{proof}
The equivalence $(1)\Leftrightarrow (2)$ is due to corollary \ref{02072025}, while the implication $(3)\Rightarrow (2)$ is clear.

Now assume (1) for the rest of the proof.

We begin by showing that $ \Int{ \Cl{A} } EA$, that is
\[
\Int{ \Cl{A} } \setminus A,A\setminus \Int{ \Cl{A} } \in \mathcal N (X)
\]

To establish that $\Int{ \Cl{A} } \setminus A\in \mathcal N (X)$, it is enough to observe that $ \Int{ \Cl{A} } \setminus A\subseteq \Cl{A} \setminus \Int{A} = \Fr{A} $, for any set $A$, and then use proposition \ref{prop:beg}(4).
Similarly $A\setminus \Int{ \Cl{A} } \subseteq A\setminus \Int{A} \subseteq \Fr{A} $, implying $A\setminus \Int{ \Cl{A} } \in \mathcal N (X)$.

Next step is to show that in any $E$-class there is at most one regular open set.
So let $U,V$ be regular open and such that $UEV$, that is
\[
U\setminus V,V\setminus U\in \mathcal N (X)
\]
From $U\setminus V\in \mathcal N (X)$ it follows that $U\subseteq \Cl{V} $, otherwise the non-empty, open $U\setminus \Cl{V}\subseteq U\setminus V$ would lead to a contradiction, so $U\subseteq \Int{ \Cl{V}} =V$.
Similarly, $V\subseteq U$ whence $U=V$.

In particular, we have that $ \Int{ \Cl{A} } $ is the same for all monitorable sets $A$ in a given $E$-class.
As $OE \Int{ \Cl{O} } \supseteq O$ for every open set $O$, the last assertion is also obtained.
\end{proof}

\begin{cor}
\leavevmode
\begin{enumerate}
\item For every $M\in \mathcal M (X)$ there exists a unique $(O,N)\in \mathcal R \mathcal O (X)\times \mathcal N (X)$ such that $M=O\triangle N$
\item The quotient Boolean algebra $ \quot{ \mathcal M (X)}{ \mathcal N (X)} $ is isomorphic to the regular open algebra $ \mathcal R \mathcal O (X)$
\end{enumerate}
\end{cor}

\begin{proof}
(1) By propositon \ref{05072023}(3) let $O$ be the unique regular open set such that $M=O\triangle N$ for some nowhere dense set $N$.
Then $N=M\triangle O$ is unique as well.

(2) The function $[A]_E\mapsto \Int{ \Cl{A} } $ is an isomorphism $ \quot{ \mathcal M (X)}{ \mathcal N (X)} \to \mathcal R \mathcal O (X)$ of Boolean algebras.
\end{proof}

\section{Monitorable sets in countable spaces}
 \label{sec:cntbl}

The purpose of this section is to discuss the descriptive set theoretic complexity of $ \mathcal M (X)$ as a subset of $ \mathcal P (X)$ when $X$ is a countable space. 
We first recall basic notions  in descriptive set theory (Section~\ref{sec:basics-top}). 
Then we describe the possible complexities of the family of monitorable sets in second countable spaces (Section~\ref{sec:sccnt}), and apply these results to topologies induced by transition relations (Section~\ref{sec:trn}). 
Finally, we show by an example that for non-second countable spaces the situation is far more complicated (Section~\ref{sec:nsc}). 

\subsection{Preliminaries on descriptive set theory}
\label{sec:basics-top}

We recall basic notions in descriptive set theory that we will use in the rest of this section, referring to \cite{Kechris1995} for a discussion in full detail.

Given a topological space $Z$, denote $ \bsigma^0_1(Z)$ the open family of $Z$, and $ \bpi^0_1(Z)$ the closed family of $Z$.
By recursion, for $2\leqslant\alpha <\omega_1$ let $ \bsigma^0_{\alpha }(Z)$ be the family of at most countable unions of members of $\bigcup_{\beta <\alpha } \bpi^0_{\beta }(Z)$, and set $ \bpi^0_{\alpha }(Z)=\{ Z\setminus A\mid A\in \bsigma^0_{\alpha }(Z)\} $.
For $1\leqslant\alpha <\omega_1$ let also $ \bdelta^0_{\alpha }(Z)= \bsigma^0_{\alpha }(Z)\cap \bpi^0_{\alpha }(Z)$.
The $\sigma $-algebra $B(Z)=\bigcup_{1\leqslant\alpha <\omega_1} \bsigma^0_{\alpha }(Z)=\bigcup_{1\le\alpha <\omega_1} \bpi^0_{\alpha }(Z)$ is the Borel family of $Z$.
Note that the classes $ \bpi^0_{\alpha }, \bsigma^0_{\alpha }, \bdelta^0_{\alpha }$ are \emph{boldface pointclasses}, that is, they are closed under continuous preimages.

A \emph{Polish} space is a separable, completely metrisable topological space.
If $Z$ is Polish, then $ \bsigma^0_{\alpha }(Z)\cup \bpi^0_{\alpha }(Z)\subseteq \bdelta^0_{\beta }(Z)$ whenever $\alpha <\beta $.
A subset $A$ of $Z$ is \emph{analytic}, or $ \bsigma^1_1$, if there exist a Polish space $Y$ and a continuous function $f:Y\to Z$ whose range is $A$; it is \emph{coanalytic}, or $ \bpi^1_1$, if its complement is analytic.

We will consider Polish spaces of the form $ \mathcal P (X)$ where $X$ is a countable topological space: indeed, identifying $ \mathcal P (X)$ with $2^X$ (via the identification of every subset with its characteristic function) and endowing the latter with the product topology where $2$ is discrete, $ \mathcal P (X)$ is homeomorphic to the Cantor space.

If $ \boldsymbol \Gamma $ is a boldface pointclass and $Z$ a Polish space, a set $A\subseteq Z$ is $\boldsymbol \Gamma $-\emph{hard} if for every zero-dimensional Polish space $Y$ and every $B\in \boldsymbol \Gamma (Y)$, the set $B$ is \emph{Wadge reducible} to the set $A$, that is, there exists a continuous function $f:Y\to Z$ such that $B=f^{-1}(A)$.
The set $A\subseteq Z$ is $ \boldsymbol \Gamma $-complete if $A\in \boldsymbol \Gamma (Z)$ and $A$ is $ \boldsymbol \Gamma $-hard.
A way to prove that a set is $ \boldsymbol \Gamma $-hard is to show that some $ \boldsymbol \Gamma $-hard set Wadge reduces to it.

If $Z$ is a Polish space, then every set in $ \bsigma^0_{\alpha }(Z)\setminus \bpi^0_{\alpha }(Z)$ is $ \bsigma^0_{\alpha }$-complete, and every set in $ \bpi^0_{\alpha }(Z)\setminus \bsigma^0_{\alpha }(Z)$ is $ \bpi^0_{\alpha }$-complete.
Every set in $ \bdelta^0_1(Z)\setminus\{\emptyset ,Z\} $ is $ \bdelta^0_1$-complete.
On the other hand, if $\alpha\geqslant 2$ then no set in $ \bdelta^0_{\alpha }(Z)$ is $ \bdelta^0_{\alpha }$-complete; we say that $A\subseteq Z$ is $ \bdelta^0_{\alpha }$-\emph{sharp} if $A\in \bdelta^0_{\alpha }(Z)\setminus\bigcup_{\beta <\alpha }( \bsigma^0_{\beta }(Z)\cup \bpi^0_{\beta }(Z))$.

\subsection{The second countable case} \label{sec:sccnt}
In this section we determine the sharp complexities of the sets $ \mathcal M (X)$, for $X$ ranging over countable, second countable topological spaces.

\begin{prop} \label{countseccount}
Let $X$ be a countable, second countable topological space.
Then $ \mathcal M (X)$ is a $ \bpi^0_3$ subset of $ \mathcal P (X)$.
\end{prop}

\begin{proof}
Fix an at most countable basis $ \mathcal B $ for $X$.
The subformula $U\subseteq \Cl{A} $ in \eqref{eq:charact} means $\forall V\in \mathcal B \ (V\subseteq U\Rightarrow V\cap A\ne\emptyset )$, that is, $\forall V\in \mathcal B \ (V\subseteq U\Rightarrow\exists x\in X\ (x\in V\land x\in A))$; similarly, the subformula $U\subseteq \Cl{X\setminus A} $ means $\forall V\in \mathcal B \ (V\subseteq U\Rightarrow\exists x\in X\ (x\in V\land x\notin A))$.
Both subformulas define a $ \bpi^0_2$ set.
Therefore \eqref{eq:charact} yields that $ \mathcal M (X)$ is $ \bpi^0_3$, due to the prefix $\neg\exists U\in \mathcal B$.
\end{proof}

\noindent
The following characterises the second countable topologies, no matter the cardinality, in which all sets are monitorable.

\begin{prop} \label{06072025}
Let $X$ be a second countable space.
Then $ \mathcal M (X)= \mathcal P (X)$ if and only if the set of isolated points is dense in $X$.
\end{prop}

\begin{proof}
Assume first that the set of isolated points is dense and let $U$ be a non-empty open subset of $X$.
Since $U$ contains an isolated point, no set can be dense and codense in $U$, and every subset of $X$ is monitorable.

Conversely, assume that the set of isolated points is not dense in $X$, and let $V$ be a non-empty, open subset that contains no isolated points.
If $V$ contains a non-empty, open set $U$ that is minimal with respect to inclusion, this must have at least $2$ elements; then any set $A$ such that $\emptyset\ne A\cap U\ne U$ is dense and codense in $U$, therefore non-monitorable.
Otherwise every non-empty, open set included in $V$ must be infinite.
Let $\{ U_n\mid n\in \N \} $ be an open basis for the induced topology of $V$.
Define by recursion distinct elements $x_n,y_n\in U_n\setminus\bigcup_{m<n}\{ x_m,y_m\} $: the set $\{ x_n\mid n\in \N \} $ is dense and codense in $V$, so it is not monitorable.
\end{proof}

\noindent
Proposition \ref{06072025} tells when $ \mathcal M(X)$ has the lowest possible complexity, namely $ \mathcal M (X)= \mathcal P (X)$.
The other complexities realised by $ \mathcal M (X)$ for $X$ countable and second countable are described in the following example.

\begin{exa} \label{ex:exmpls}
\leavevmode
\begin{enumerate}
\item $ \mathcal M (X)$ can be a complete clopen set.

Let $X$ be the topological sum of a countable discrete space $Y$ and a $2$-element indiscrete space $Z=\{ z_1,z_2\} $, so that $ \mathcal M (X)\ne \mathcal P (X)$ by proposition \ref{06072025}.
Given $A\subseteq X$, we have that $A$ is not monitorable if and only if $\emptyset\ne A\cap Z\ne Z$, that is, if and only if $(z_1\in A\lor z_2\in A)\land (z_1\notin A\lor z_2\notin A)$, which is a clopen condition on $A$.
\item $ \mathcal M (X)$ can be a complete closed set.

In fact, $ \mathcal M (X)=\{\emptyset , X\} $ if and only if $X$ is an indiscrete topological space.
\item $ \mathcal M (X)$ can be a $ \bsigma^0_2$-complete set.

Let $X$ be a countable set endowed with the cofinite topology.
A subset of $X$ is dense and codense in some non-empty, open set if and only if it is dense and codense if and only if it is infinite and coinfinite.
This is a $ \bpi^0_2$ condition, so $ \mathcal M (X)$ is $ \bsigma^0_2$.
To see completeness, just observe that $ \mathcal M (X)$ is dense and codense in $ \mathcal P (X)$, so $ \mathcal M (X)$ cannot be $ \bpi^0_2$, by the Baire category theorem.
\item $ \mathcal M (X)$ can be a $ \bpi^0_3$-complete set.

Let $X= \N^2 $ be partially ordered by letting
\[
(m,n)\preccurlyeq (m',n')\Leftrightarrow m=m'\land n\leqslant n'
\]
and endow $X$ with the Alexandrov topology induced by this partial order.
Then $ \mathcal M (X)$ is a $ \bpi^0_3$-complete subset of $ \mathcal P (X)$.

Indeed, let $S_3=\{\alpha\in 2^{ \N^2}\mid\exists m\in \N \ \exists^{\infty }n\in \N \ \alpha (m,n)=0\} $, which is $ \bsigma^0_3$-complete (see \cite[\S 23.A]{Kechris1995}).
Define $f:2^{ \N^2}\to \mathcal P (X)$ by letting
\[
f(\alpha )=\{ (m,2n)\mid\alpha (m,n)=0\}
\]
The function $f$ is continuous.
Moreover, $f(\alpha )$ is codense in $X$ for every $\alpha $, as $\{ (m,2n+1)\mid m,n\in \N \}\subseteq X\setminus f(\alpha )$.
Finally:
\begin{itemize}
\item If $\alpha\in S_3$, let $m\in \N $ be such that $\exists^{\infty }n\in \N \ \alpha (m,n)=0$.
Then $f(\alpha )$ is dense in the open set $\{ (m,n)\mid n\in \N \} $, so that $f(\alpha )\notin \mathcal M (X)$.
\item If $\alpha\notin S_3$, consider any non-empty, open $U\subseteq X$. There exist $m,n_0\in \N $ such that $\{ (m,n)\mid n\geqslant n_0\}\subseteq U$; let also $n_1\geqslant \frac{n_0}2 $ be such that $\forall n\geqslant n_1\ \alpha (m,n)=1$.
Then $\{ (m,n)\mid n\geqslant 2n_1\} $ is an open subset of $U$, and $f(\alpha )\cap\{ (m,n)\mid n\geqslant 2n_1\} =\emptyset $, witnessing that $f(\alpha )$ is not dense in $U$.
This shows that $f(\alpha )\in \mathcal M (X)$.
\end{itemize}
Therefore $2^{ \N^2}\setminus S_3=f^{-1}( \mathcal M (X))$, establishing the claim.
\end{enumerate}
\end{exa}

\noindent
In the rest of this section we show that the only possible complexities for $ \mathcal M (X)$ are those given by proposition \ref{06072025} and example \ref{ex:exmpls}.
In particular, we derive a procedure to recognise when $ \mathcal M (X)$ is $ \bpi^0_3$-complete.

\begin{thm} \label{thm:wishful}
Let $X$ be a countable, second countable space.
If $ \mathcal M(X)\notin \bpi^0_1( \mathcal P (X))$, then $ \mathcal M (X)$ is $ \bsigma^0_2$-hard.
\end{thm}

\noindent
The proof of theorem \ref{thm:wishful} is obtained through a series of lemmas.

Assume that $ \mathcal M (X)$ is not closed in $ \mathcal P (X)$ and let $B\in \mathcal P (X)\setminus \mathcal M (X)$ be such that there exists a sequence of sets $A_n\in \mathcal M (X)$ with $\lim_{n\rightarrow\infty }A_n=B$.
Let $V$ be a non-empty, open subset of $X$ in which $B$ is dense and codense.

\begin{lem} \label{lem:nominopenset}
There is no minimal non-empty, open subset of $V$.
\end{lem}

\begin{proof}
Assume towards contradiction that $U$ is a minimal non-empty, open subset of $V$.
Then $\emptyset\ne B\cap U\ne U$, so let $x\in B\cap U,y\in U\setminus B$.
So eventually $x\in A_n,y\notin A_n$, but then $A_n$ is dense and codense in $U$, a contradiction.
\end{proof}

\noindent
In particular, the following holds.

\begin{cor} \label{lem:openinfinite}
All non-empty, open subsets of $V$ are infinite.
\end{cor}

\noindent
Suppose there exists $x^*\in V$ such that $ \Int{ \Cl{ \{ x^*\} } } \ne\emptyset $.
Note that every non-empty, open subset of $ \Int{ \Cl{ \{ x^*\} } } $ is a neighbourhood of $x^*$.
By lemma \ref{lem:nominopenset}, let $\{ U_n\mid n\in \N \} $ be a strictly decreasing basis of open neighbourhoods of $x^*$, with $U_0\subseteq \Int{ \Cl{ \{ x^*\} } } $.
For every $n\in \N $ pick $x_n\in U_n\setminus U_{n+1}$.

Let $f: \Can \to \mathcal P (X)$ be defined by letting $f(\alpha )=\{ x_n\mid\alpha (n)=1\} $.
Then $f$ is continuous and $f(\alpha )\in \mathcal M (X)\Leftrightarrow\forall^{\infty }n\in \N \ \alpha (n)=0$.
Indeed, $f(\alpha )$ is codense in $ \Int{ \Cl{ \{ x^*\} } } $ as $x^*\notin f(\alpha )$; moreover, if $\exists^{\infty }n\in \N \ \alpha (n)=1$ then $f(\alpha )$ is also dense in $ \Int{ \Cl{ \{ x^*\} } } $.
If instead $\forall^{\infty }n\in \N \ f(\alpha )=0$, then eventually $U_n\cap f(\alpha )=\emptyset $; it follows that $f(\alpha )$ is not dense in any non-empty open set.

As $\{\alpha\in \Can \mid\forall^{\infty }n\in \N \ \alpha (n)=0\} $ is $ \bsigma^0_2$-complete by \cite[\S 23.A]{Kechris1995}, $ \mathcal M (X)$ is $ \bsigma^0_2$-hard.

Therefore, it can be assumed that
\begin{equation} \label{eq:intemp}
\forall x\in V\ \Int{ \Cl{ \{ x\} } } =\emptyset
\end{equation}

\begin{lem} \label{lem:finsmot}
If $A$ is finite and $A\subseteq V$, then $A\in \mathcal M (X)$.
\end{lem}

\begin{proof}
Let $U$ be non-empty and open.
Then $U\setminus \Cl{A} $ is non-empty by \eqref{eq:intemp} and open, so $A$ is not dense in $U$.
\end{proof}

\noindent
Let $\{ U_n\mid n\in \N \} $ be an enumeration without repetitions of an open basis for the induced topology of $V$.
Let also $\{ n_m\mid m\in \N \} $ be an enumeration of $ \N $ with every natural number appearing infinitely many times.
For every $m\in \N $ pick recursively distinct elements $x_m,y_m\in U_{n_m}\setminus\bigcup_{h<m}\{ x_h,y_h\} $: this is possible by corollary \ref{lem:openinfinite}.
Finally, for every $n\in \N $ let $\{ u_{nk}\mid k\in \N \} $ be an enumeration without repetition of $\{ x_m\mid U_{n_m}=U_n\} $.

Define $f: \Can \to \mathcal P (X)$ by letting $f(\alpha )=\bigcup_{\alpha (h)=1}\{ u_{0h},u_{1h},\ldots ,u_{hh}\} $, which is codense in $V$ as it misses $\{ y_m\mid m\in \N \} $.
The function $f$ is continuous.
Moreover:
\begin{itemize}
\item if $\{ h\in \N \mid\alpha (h)=1\} $ is finite, then $f(\alpha )\in \mathcal M (X)$ by lemma \ref{lem:finsmot}
\item if $\{ h\in \N \mid\alpha (h)=1\} $ is infinite, then $f(\alpha )$ is dense in $V$, so $f(\alpha )\notin \mathcal M (X)$
\end{itemize}

\noindent
This reduces the $ \bsigma^0_2$-complete set $\{\alpha\in \Can \mid\forall^{\infty }n\in \N \ \alpha (n)=0\} $ to $ \mathcal M (X)$, concluding the proof of theorem \ref{thm:wishful}.

\begin{cor}
If $X$ is a countable, second countable space then $ \mathcal M (X)$ cannot be $ \bsigma^0_1$-complete, nor $ \bdelta^0_2$-sharp, nor $ \bpi^0_2$-complete.
\end{cor}

\noindent
It remains to deal with the last possible case, namely $ \bpi^0_3$-completeness.

Fix a countable, second countable space $X$.
Let $\{ U_n\}_n$ be an enumeration of all minimal non-empty, open sets, and let $H=\bigcup_nU_n$.
Let $L= \Int{X\setminus H} $.

\begin{lem} \label{eq:conddec}
Let $A\in \mathcal P (X)$.
Then
\[
A\in\mathcal M (X)\Leftrightarrow\forall n\ (U_n\subseteq A\lor A\cap U_n=\emptyset )\land A\cap L\in \mathcal M (L)
\]
\end{lem}

\begin{proof}
If $A$ is monitorable then for every $n$ either $A\cap U_n=\emptyset $ or $U_n\subseteq A$, since every non-empty, proper subset of $U_n$ is dense and codense in $U_n$; moreover $A$ cannot be both dense and codense in some non-empty, open subset of $L$, so $A\cap L\in \mathcal M (L)$.

Conversely, assume that $A$ is not monitorable and let $U$ be non-empty, open, and such that $A$ is dense and codense in $U$.
If $U\subseteq L$, then $A\cap L$ is not monitorable in $L$.
Otherwise there exists $n$ such that $U_n\subseteq U$.
It follows that $\emptyset\ne A\cap U_n\ne U_n$.
\end{proof}

\begin{lem} \label{lem:verymany}
For every non-empty, open $U\subseteq L$ there exist infinitely many $x\in U$ such that $ \Int{ \Cl{\{ x\} } } =\emptyset $.
\end{lem}

\begin{proof}
Deny and let $\{ x_1,\ldots ,x_r\} =\{ x\in U\mid \Int{ \Cl{\{ x\} } } =\emptyset\} $.
As $ \Int{ \Cl{ \{ x_1,\ldots ,x_r\} } } =\emptyset $, let $V=U\setminus \Cl{ \{ x_1,\ldots ,x_r\} } =U\setminus\bigcup_{s=1}^r \Cl{\{ x_s\} } \ne\emptyset $.
Then $\forall x\in V\ \Int{ \Cl{\{ x\} } } \ne\emptyset $.
Fix any $x\in V$ and let $W= \Int{ \Cl{\{ x\} } } \cap V\ne\emptyset $.

We claim that $\forall y\in W\ W\subseteq \Int{ \Cl{\{ y \}} } $.
Indeed $x\in \Int{ \Cl{\{ y\} } } $, since $ \Int{ \Cl{\{ y\} } }$ is an open set containing $y$ and $y\in \Cl{ \{ x\} } $, hence $W\subseteq \Int{ \Cl{\{ x\}} } \subseteq \Int{ \Cl{ \Int{ \Cl{ \{ y\} } } } } = \Int{ \Cl{\{ y\} } } $.

As every singleton subset of $W$ is dense in $W$, it follows that $W$ is a minimal non-empty, open set, which is a contradiction.
\end{proof}

\begin{lem} \label{lem:conditionpi}
If there exist infinitely many pairwise disjoint open subsets of $L$, then $ \mathcal M (X)$ is $ \bpi^0_3$-complete.
\end{lem}

\begin{proof}
Let $\{ V_r\mid r\in \N \} $ be an enumeration, without repetitions, of a family of pairwise disjoint, non-empty, open subsets of $L$, and for every $r\in \N $ let $\{ V_{rn}\mid n\in \N \} $ be an enumeration without repetitions of a basis of the induced topology on $V_r$.
Let also $\{ n_m\mid m\in \N \} $ be an enumeration of $ \N $ with every element repeated infinitely many times.
For every $r\in \N $ pick, by recursion on $m$, distinct elements $x_{rm},y_{rm}\in V_{rn_m}\setminus\bigcup_{h<m}\{ x_{rh},y_{rh}\} $, with $ \Int{ \Cl{\{ x_{rm}\} } } =\emptyset $, which is possible by lemma \ref{lem:verymany}.
Finally, for every $r,n\in \N $ let $\{ u_{rnk}\mid k\in \N \} $ be an enumeration without repetitions of $\{ x_{rm}\mid V_{rn_m}=V_{rn}\} $.

Define $f:2^{ \N^2}\to \mathcal P (X)$ by letting $f(\alpha )=\bigcup_{\alpha (r,h)=1}\{ u_{r0h},u_{r1h},\ldots ,u_{rhh}\} $, which is codense in $L$.
The function $f$ is continuous.
Moreover:
\begin{itemize}
\item if $\forall r\in \N \ \forall^{\infty }h\in \N \ \alpha (r,h)=0$, then $ \Int{ \Cl{f(\alpha )} } =\emptyset $, since every $f(\alpha )\cap V_r$ is a finite union of singletons whose closures have empty interior; so $f(\alpha )\in \mathcal M (X)$
\item if $\exists r\in \N \ \exists^{\infty }h\in \N \ \alpha (r,h)=1$, then $f(\alpha )$ is dense in $V_r$, so $f(\alpha )\notin \mathcal M (X)$
\end{itemize}
Therefore the $ \bpi^0_3$-complete set $\{\alpha\in 2^{ \N^2}\mid\forall r\in \N \ \forall^{\infty }h\in \N \ \alpha (r,h)=0\} $ Wadge reduces to $ \mathcal M (X)$, proving that $ \mathcal M (X)$ is $ \bpi^0_3$-complete.
\end{proof}

\noindent
Therefore assume that there is no infinite family of pairwise disjoint open subsets of $L$.

\begin{lem}
There exists $N\in \N $ such that every family of pairwise disjoint open subsets of $L$ has cardinality at most $N$.
\end{lem}

\begin{proof}
Note that if $ \mathcal W $ is a family of non-empty, pairwise disjoint, open subsets of a topological space, then $ \mathcal W $ is maximal if and only if $\bigcup \mathcal W $ is dense.

Assume towards contradiction that there exist families of arbitrarily high finite cardinality of pairwise disjoint open subsets of $L$.
We define by recursion a sequence of maximal families $ \mathcal W_n$ of non-empty, pairwise disjoint, open subsets of $L$, with each $ \mathcal W_{n+1}$ refining $ \mathcal W_n$.

Let $ \mathcal W_0=\{ L\} $.

Assume $ \mathcal W_n$ is defined, and let $ \mathcal V $ be a maximal family of non-empty, pairwise disjoint, open subsets of $L$, with $ \card ( \mathcal V)> \card ( \mathcal W_n)$.
Set $ \mathcal W_{n+1}=\{ W\cap V\mid W\in \mathcal W_n,V\in \mathcal V ,W\cap V\ne\emptyset\} $; notice that $ \mathcal W_{n+1}$ is maximal and $ \card ( \mathcal W_{n+1})\geqslant \card ( \mathcal V )> \card ( \mathcal W_n)$.

Let $ \mathcal T $ be the tree whose nodes are the elements of $\bigcup_{n\in \N } \mathcal W_n$ with the order of reverse inclusion.
Since there are infinitely many branching nodes and $ \mathcal T $ is finite splitting, let $b$ be a branch of $ \mathcal T $ containing infinitely many of them.
If $W$ is a branching point belonging to $b$, let $W^+$ be an immediate successor of $W$ that does not belong to $b$.
Then the family of all such $W^+$ is infinite and consists of pairwise disjoint open subsets of $L$, a contradiction.
\end{proof}

\noindent
Therefore let $ \mathcal W =\{ W_1,\ldots ,W_N\} $ be a family of maximal cardinality of non-empty, pairwise disjoint, open subsets of $L$.
Notice that each $W_h$ is a \emph{hyperconnected} space, meaning that there are no two disjoint, open, non-empty subsets of $W_h$.

\begin{lem} \label{lem:newcorr}
If $Y$ is a countable, second countable, hyperconnected space, then $ \mathcal M (Y)\in \bsigma^0_2( \mathcal P (Y))$.
\end{lem}

\begin{proof}
Notice that a subset $A$ of $Y$ is dense in $Y$ if and only if it is dense in some (equivalently, any) non-empty, open subset of $Y$.
Indeed, the forward implication holds for any topological space.
For the converse, fix $U$ non-empty and open in $Y$ and assume that $A$ is dense in $U$.
Let $V$ be any non-empty open set.
Then $\emptyset\ne A\cap U\cap V\subseteq A\cap V$.

Therefore, given any $A\subseteq Y$,
\begin{align*}
A\notin \mathcal M (Y)\Leftrightarrow & A \text{ is dense and codense in } Y\Leftrightarrow \\
\Leftrightarrow & \text{for every } U \text{ basic open set},A\cap U\ne\emptyset\ne U\setminus A
\end{align*}
which is a $ \bpi^0_2$-condition.
\end{proof}

\noindent
Now, given $A\in \mathcal P (X)$, observe that $A\cap L\notin \mathcal M (L)\Leftrightarrow\exists h\in\{ 1,\ldots ,N\}\ A\cap W_h\notin \mathcal M (W_h)$.
This is a $ \bpi^0_2$ condition by lemma \ref{lem:newcorr}, hence $ \mathcal M (X)\in \bsigma^0_2( \mathcal P (X))$ by lemma \ref{eq:conddec}.

\begin{cor} \label{23112025}
Let $X$ be a countable, second countable space.
If $ \mathcal M (X)\notin \bsigma^0_2( \mathcal P (X))$, then $ \mathcal M(X)$ is $ \bpi^0_3$-complete.

In particular, $ \mathcal M (X)$ cannot be $ \bdelta^0_3$-sharp.
\end{cor}

Figure \ref{fig:res} summarises the positions in the Borel hierarchy that are occupied by families of monitorable sets.

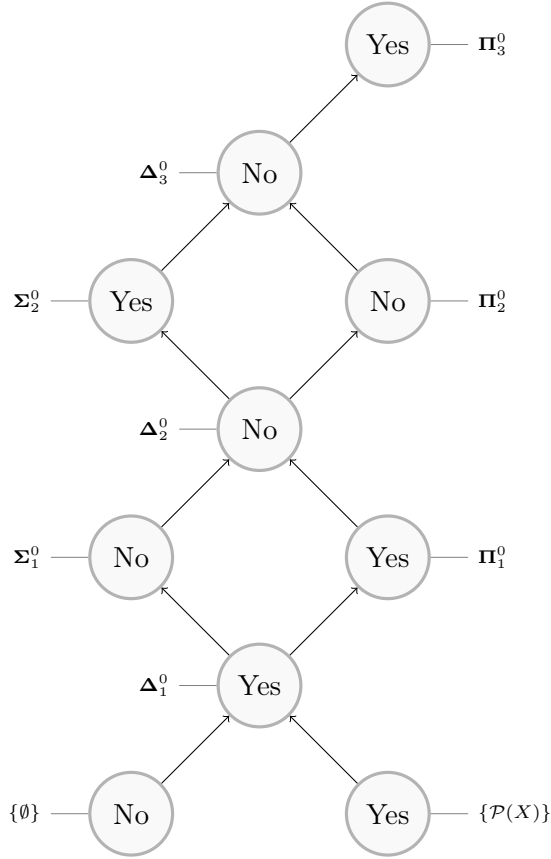
\begin{figure}
\begin{tikzpicture}[
roundnode/.style={circle, draw=gray!60, fill=gray!5, very thick, minimum size=11mm},
]
\node[roundnode]      (empty) [pin=180:{\tiny $\{\emptyset \} $}]                             {No};
\node (void0) [right=of empty] {};
\node[roundnode]      (all)       [right=of void0] [pin=0:{\tiny $\{ \mathcal P (X)\} $}] {Yes};
\node[roundnode]      (clopen)       [above=of void0] [pin=180:{\tiny $ \bdelta^0_1$}] {Yes};
\node (void1)       [above=of empty] { };
\node (void2)       [above=of all] { };
\node[roundnode] (open) [above=of void1] [pin=180:{\tiny $ \bsigma_1^0$}] {No};
\node[roundnode]      (closed)       [above=of void2] [pin=0:{\tiny $ \bpi^0_1$}] {Yes};
\node      (void3)       [above=of clopen] { };
\node[roundnode]      (delta02)       [above=of void3] [pin=180:{\tiny $ \bdelta_2^0$}] {No};
\node (void4) [above=of open] { };
\node (void5) [above=of closed] { };
\node[roundnode] (sigma02) [above=of void4] [pin=180:{\tiny $ \bsigma_2^0$}] {Yes};
\node[roundnode]      (pi02)       [above=of void5] [pin=0:{\tiny $ \bpi^0_2$}] {No};
\node (void6) [above=of delta02] {};
\node[roundnode] (delta03) [above=of void6] [pin=180:{\tiny $ \bdelta_3^0$}] {No};
\node (void7) [above=of pi02] { };
\node[roundnode]      (pi03)       [above=of void7] [pin=0:{\tiny $ \bpi^0_3$}] {Yes};
\draw[->] (empty) -- (clopen);
\draw[->] (all) -- (clopen);
\draw[->] (clopen) -- (open);
\draw[->] (clopen) -- (closed);
\draw[->] (open) -- (delta02);
\draw[->] (closed) -- (delta02);
\draw[->] (delta02) -- (sigma02);
\draw[->] (delta02) -- (pi02);
\draw[->] (sigma02) -- (delta03);
\draw[->] (pi02) -- (delta03);
\draw[->] (delta03) -- (pi03);
\end{tikzpicture}
\caption{The allowed positions of $ \mathcal M (X)$ in the Borel hierarchy, for $X$ a countable, second countable space}
\label{fig:res}
\end{figure}

\begin{rem} \label{231120251728}
The construction leading to corollary \ref{23112025} provides an operational way to recognise whether $ \mathcal M (X)$ is $ \bsigma^0_2$ or $ \bpi^0_3$-complete, for $X$ a countable, second countable topological space: starting from $X$, remove the union $H$ of all minimal non-empty, open sets; if inside $X\setminus H$ (equivalently, inside $L= \Int{X\setminus H} $) there are infinitely many pairwise disjoint open subsets of $X$, then $ \mathcal M (X)$ is $ \bpi^0_3$-complete; otherwise it is $ \bsigma^0_2$.
\end{rem}

\begin{exa}
Consider the space $E^{{<}\omega}$ of Example~\ref{ex:finstr} and let $a,b\in E$ be two distinct elements. 
Note that in this space there are no minimal open subsets. 
Moreover, the set 
$B = \{ (a^nb)^{\uparrow }\mid n \in \N \}$ 
is infinite and consists of pairwise disjoint open subsets of $E^{{<}\omega}$. 
Therefore, by the criterion of Remark~\ref{231120251728}, we get that $\mathcal{M}(E^{{<}\omega})$ is $\bpi^0_3$-complete. 
\end{exa}

\subsection{An application: transition relations}
 \label{sec:trn}
The results of section \ref{sec:sccnt} show that for a second countable topology on a countable set $X$ the family of monitorable subsets can be either simple (namely, a $ \bsigma^0_2$ subset of $ \mathcal P (X)$) or complicated (namely, a $ \bpi^0_3$-complete subset of $ \mathcal P (X)$).
Moreover, Remark \ref{231120251728} yields a procedure to distinguish the two situations.
In this section, we apply this procedure to a class of topologies induced by transition relations on a countable set, 
proving that most of them give rise to a simple family of monitorable subsets 
(Theorem \ref{thm:trn}). 

Let $X$ be a countable set, and let $E$ be a non-empty, at most countable set.
Recall that a transition relation on $X$ is a subset $R \subseteq X \times E \times X$. 
Then, for every $R\subseteq X\times E\times X$, we have a topology $ \mathcal T_R$ on $X$, generated by the family $ \mathcal B_R=\{ U_F^R\mid F \text{ a finite subset of } E^{<\omega }\} $ where
\[\begin{split} 
U_F^R = \{ q_0\in X\mid &\forall s\in F\ \exists q_1,\ldots ,q_{ \lnh (s)}\in X \\
  &\forall h\in\{ 0,\ldots , \lnh (s)-1\}\ (q_h,s(h),q_{h+1})\in R \}
\end{split}\] 
The family $ \mathcal B_R$ is obviously at most countable and indeed a basis, as $U_F^R\cap U_G^R=U_{F\cup G}^R$.
Therefore $ \mathcal T_R$ is second countable.

\begin{rem}
This topology has been considered in \cite{CicconeDF25} (even if the authors do not use explicitly a topological language) 
 as a semantic framework for runtime verification of branching-time properties.
Intuitively, the set $X$ is the state space of a system under analysis and the elements of $E$ are the possible observable events. 
Then, the topology $\mathcal T_R$ abstractly models monitors that can observe finitely many finite sequences of events generated by system executions. 
It is worth noticing that, as discussed in \cite{CicconeDF25}, this setting is quite different from others used in the literature \cite{FrancalanzaAI17}, which however use also a different notion of monitorability. 
\end{rem}

\noindent
Let us endow the set $\mathcal{P}(X\times E \times X)$ with the product topology, after identifying each transition relation with its characteristic function. 
We consider the following two subsets of $\mathcal{P}(X\times E \times X)$: 
\begin{align*}
\mathcal R_0= & \{ R\in \mathcal P (X\times E\times X)\mid \mathcal M ( \mathcal T_R)\in \bsigma^0_2( \mathcal P (X))\} \\
\mathcal R_1= & \{ R\in \mathcal P (X\times E\times X)\mid \mathcal M ( \mathcal T_R) \text{ is } \bpi^0_3 \text{-complete} \} = \\
= & \mathcal P (X\times E\times X)\setminus \mathcal R_0
\end{align*}
Therefore, $\mathcal R_0$ collects those transition relations inducing a simple family of monitorable subsets, while $\mathcal R_1$ those whose family of monitorable subsets is complex. 
Recall that a subset $A$ of a topological space is \emph{meagre} if it is an at most countable union of nowhere dense sets, while it is \emph{comeagre} if its complement is meagre.

\begin{thm} \label{thm:trn}
\leavevmode
\begin{enumerate}
\item\label{thm:trn:1} Both $ \mathcal R_0, \mathcal R_1$ are Borel subsets of $ \mathcal P (X\times E\times X)$
\item\label{thm:trn:2} The set $ \mathcal R_0$ is comeagre in $ \mathcal P (X\times E\times X)$; equivalently, the set $ \mathcal R_1$ is meagre in $ \mathcal P (X\times E\times X)$
\end{enumerate}
\end{thm}

\begin{proof}
(\ref{thm:trn:1})  Given $q\in X$ and a finite subset $F$ of $E^{<\omega }$, the set $\{ R\in \mathcal P (X\times E\times X)\mid q\in U_F^R\} $ is open.
To see that $ \mathcal R_1$ is Borel, notice that, by the characterisation of section \ref{sec:sccnt}, one has $R\in \mathcal R_1$ if and only if for every $N\in \N $ there exist finite subsets $F_0,\ldots ,F_N$ of $E^{<\omega }$ such that
\begin{itemize}
\item every $U_{F_h}^R$ is non-empty
\item if $h\ne k$ then $U_{F_h}^R\cap U_{F_k}^R=\emptyset $
\item no $U_{F_h}^R$ contains a minimal non-empty, open set; this means that for every $F$ finite subset of $E^{<\omega }$ the following implication holds:
\[
\emptyset\ne U_F^R\subseteq U_{F_h}^R\Rightarrow\exists G\ \emptyset\ne U_G^R\subset U_F^R
\]
\end{itemize}

\noindent
As the listed conditions are Borel, $ \mathcal R_1$ is Borel, whence $ \mathcal R_0$ is Borel as well.

(\ref{thm:trn:2}) To prove that $ \mathcal R_0$ is comeagre in $ \mathcal P (X\times E\times X)$ we show that it contains a dense $ \bpi^0_2$ set, namely $ \mathcal H =\{ R\in \mathcal P (X\times E\times X)\mid \mathcal T_R \text{ is hyperconnected} \} $ (see lemma \ref{lem:newcorr}).

Note that $R\in \mathcal H $ if and only if, for every $F,G$ finite subsets of $E^{<\omega }$, there exists $q\in U_F^R\cap U_G^R=U_{F\cup G}^R$.
Since the condition $q\in U_{F\cup G}^R$ is open in the variable $R$, it follows that $ \mathcal H \in \bpi^0_2( \mathcal P (X\times E\times X))$.

To show that $ \mathcal H $ is dense, fix a basic open set
\[ \begin{array}{l}
\mathcal U =\{ R\in \mathcal P (X\times E\times X)\mid \\
(q_1,e_1,q'_1),\ldots ,(q_n,e_n,q'_n)\in R, (q''_1,e'_1,q'''_1),\ldots ,(q''_m,e'_m,q'''_m)\notin R\}
\end{array} \]
Let $q\in X\setminus\{ q''_1\ldots ,q''_m,q'''_1,\ldots ,q'''_m\} $ and set $R=\{ (q_1,e_1,q'_1),\ldots ,(q_n,e_n,q'_n)\}\cup\{ (q,e,q)\mid e\in E\} $.
Then $R\in \mathcal U $.
Moreover, $R\in \mathcal H$ since $q\in U_F^R$ for every finite subset $F$ of $E^{<\omega }$.
\end{proof}

\subsection{The non-second countable case} \label{sec:nsc}
The hypothesis of second countability in proposition \ref{countseccount} and in lemma \ref{lem:newcorr} cannot be removed.

Let $X= \Seq \cup\{\infty\} $ be ordered by letting
\[
x\leqslant y\Leftrightarrow (x,y\in \Seq \land x\subseteq y)\lor y=\infty
\]
and endow $X$ with the Scott topology $ \mathcal T $.
A subset of $X$ is open if and only if it is either empty or an upward closed set intersecting every infinite branch of $ \Seq $.
Note that $\infty $ belongs to all non-empty open subsets of $X$, so $X$ is a hyperconnected space.

\begin{lem} \label{lem:hypc}
Let $Y$ be a non-empty hyperconnected space and let $A\subseteq Y$.
Then the following are equivalent:
\begin{enumerate}
\item $A\in \mathcal M (Y)$
\item $ \Fr{A} \ne Y$
\item $ \Int{A} \ne\emptyset\lor \Int{Y\setminus A} \ne \emptyset $
\end{enumerate}
\end{lem}

\begin{proof}
$(1)\Rightarrow (2)$ holds in any non-empty topological space, by proposition \ref{prop:beg}(4).

$(2)\Rightarrow (3)$ holds in any topological space, since if $ \Int{A} = \Int{Y\setminus A} =\emptyset $, then given any $x\in Y$ and any neighbourhood $U$ of $x$, it follows that $U$ intersects both $A$ and $Y\setminus A$, so $x\in \Fr{A} $; consequently $ \Fr{A} =Y$.

$(3)\Rightarrow (1)$.
If $A\notin \mathcal M (Y)$, let $U$ be non-empty, open, and such that $U\subseteq \Fr{A} $.
As $U\cap \Int{A} =U\cap \Int{Y\setminus A}=\emptyset $, from the hyperconnectedness hypothesis it follows that $ \Int{A} = \Int{Y\setminus A} =\emptyset $.
\end{proof}

\noindent
If $\alpha\in \Bai $ and $n\in \N $, we denote by $\alpha_{|_n}$ the restriction of $\alpha $ to $n$, that is, the truncation of $\alpha $ to the first $n$ values.
If $t,s$ are finite sequences, we denote by $ts$ their concatenation.

\begin{lem} \label{lem:intnn}
Let $A\subseteq X$.
Then $ \Int{A} \ne\emptyset $ if and only if
\begin{equation} \label{eq:lemmaint}
\infty\in A\land\forall\alpha\in \Bai \ \exists n\in \N \ \forall s\in \Seq \ \alpha_{|_n}s\in A
\end{equation}
\end{lem}

\begin{proof}
Notice that property \eqref{eq:lemmaint} is satisfied by non-empty open sets and by any superset of a set satisfying it.
So if $ \Int{A} \ne\emptyset $, then \eqref{eq:lemmaint} holds for $A$.

Conversely, assume \eqref{eq:lemmaint}.
For every $\alpha\in \Bai $ let $n_{\alpha }\in \N $ be such that $\forall s\in \Seq\ \alpha_{|_{n_{\alpha }}}s\in A$.
Then $U=\{\infty\}\cup\bigcup_{\alpha\in \Bai }\{ t\in \Seq \mid \alpha_{|_{n_{\alpha }}}\subseteq t\}$ is open and $U\subseteq A$.
\end{proof}

\noindent
Let $Tr$ be the Polish space of trees on $ \N $ and denote by $WF$ the subset of well founded trees.
Recall that $WF$ is complete coanalytic (see \cite[\S 33.A]{Kechris1995}).

\begin{prop} \label{24112025}
The family $ \mathcal M (X)$ is a complete coanalytic subset of $ \mathcal P (X)$.
\end{prop}

\begin{proof}
To see that $ \mathcal M (X)$ is coanalytic, notice that \eqref{eq:lemmaint} is a coanalytic condition in the variable $A$; then use lemmas \ref{lem:hypc}(3) and \ref{lem:intnn}, together with the fact that the function $A\mapsto X\setminus A$ is a homeomorphism of $ \mathcal P (X)$.

For the hardness part, let $f:Tr\to \mathcal P (X)$ be defined by setting $f(T)=T$, so that $f$ is continuous.
Now notice that:
\begin{itemize}
\item If $T$ is well founded, then it is a closed subset of $X$, therefore it is monitorable.
\item If $T$ is ill founded, then $T$ is a dense subset of $X$ by the characterisation of non-empty open sets; moreover $T$ is a codense subset of $X$, since $\infty\notin T$.
Therefore it is not monitorable.
\end{itemize}
\qedhere
\end{proof}

\noindent
Proposition \ref{24112025} shows that there can be a big difference for the complexity of the family of monitorable subsets in a countable topological space, according to whether it is second countable or not.
Notice that a second countable topology on a countable space is always $ \bpi^0_3$ (see \cite[proposition 3.2(ii)]{Todorc2014}), while non-second countable topologies may be way more complicated; it can therefore be expected that the family of monitorable sets can be very complicated as well.


\bibliographystyle{alphaurl}
\bibliography{biblio}

\end{document}